\documentclass{amsart}
\usepackage{amsmath, amsthm, amscd, amsfonts, amssymb, graphicx, color}
\usepackage[bookmarksnumbered, colorlinks, plainpages]{hyperref}

 \newtheorem{thm}{Theorem}[section]
 \newtheorem{cor}[thm]{Corollary}
 \newtheorem{lem}[thm]{Lemma}
 
 \theoremstyle{definition}
 
 \theoremstyle{remark}
 \newtheorem{rem}[thm]{Remark}

\newcommand{\A}{\mathcal{A}}

\newcommand{\U}{\mathcal{U}}
\newcommand{\X}{\mathcal{X}}

\begin{document}

\title[2n-Weak module amenability of semigroup algebras]
 {2n-Weak module amenability of semigroup algebras}

\author{ Hoger Ghahramani}

\thanks{{\scriptsize
\hskip -0.4 true cm \emph{MSC(2010)}:  43A20; ; 46H25; 43A10.
}}

\address{Department of
Mathematics, University of Kurdistan, P. O. Box 416, Sanandaj,
Iran.}

\email{h.ghahramani@uok.ac.ir; hoger.ghahramani@yahoo.com}

\address{}

\email{}

\thanks{}

\thanks{}



\date{}

\dedicatory{}

\commby{}


\begin{abstract}
Let $S$ be an inverse semigroup with the set of idempotents $E$.
We prove that the semigroup algebra $\ell^{1}(S)$ is always
$2n$-weakly module amenable as an $\ell^{1}(E)$-module, for any
$n\in \mathbb{N}$, where $E$ acts on $S$ trivially from the left
and by multiplication from the right.
\end{abstract}

\maketitle

\section{Introduction}
Let $\A$ be a Banach algebra, and let $\X$ be a Banach
$\A$-bimodule. A linear map $D:\A\rightarrow \X$ is called a
\emph{derivation} if $D(ab)=aD(b)+D(a)b$ for all $a,b\in \A$. Each
map of the form $a\rightarrow ax-xa$, where $x\in\X$, is a
continuous derivation which will be called an \emph{inner
derivation}.
\par
For any Banach $\A$-module $\X$, its dual space $\X^{*}$ is
naturally equipped with a Banach $\A$-module structure via
\[ \langle x,af\rangle=\langle xa,f\rangle \quad \langle x,
fa\rangle=\langle ax, f\rangle \,\,\, (a\in \A,f\in \X^{*}, x\in
\X). \] Note that the Banach algebra $\A$ itself is a Banach
$\A$-bimodule under the algebra multiplication. So $\A^{(n)}$, the
$n$-th dual space of $\A$, is naturally a Banach $\A$-bimodule in
the above sense for each $n\in \mathbb{N}$. The Banach algebra
$\A$ is called \emph{n-weakly amenable} if every continuous
derivation from $\A$ into $\A^{(n)}$ is inner. If $\A$ is
$n$-weakly amenable for each $n\in\mathbb{N}$ then it is called
\emph{permanently weakly amenable}.
\par
The concept of $n$-weakly amenability was introduced by Dales,
Ghahramani and Gr\o nb\ae k in \cite{dal}. Johnson showed in
\cite{john1} that for any locally compact group $G$, the group
algebra $L^{1}(G)$ is always $1$-weakly amenable. It was shown
further in \cite{dal} that $L^{1}(G)$ is in fact $n$-weakly
amenable for all odd numbers $n$. Whether this is still true for
even numbers $n$ was left open in \cite{dal}. Later in
\cite{john2} Johnson proved that $\ell^{1}(G)$ is $2n$-weakly
amenable for each $n\in \mathbb{N}$ whenever $G$ is a free group.
The problem has been resolved affirmatively for general locally
compact group $G$ in \cite{choi} and in \cite{los} independently,
using a theory established in \cite{los2}. In \cite{zh}, as an
application of a common fixed point property for semigroups, a
short proof to $2m$-weak amenability of $L^{1}(G)$ was presented.
Mewomo in \cite{mew} investigate the $n$-weak amenability of
semigroup algebras and showed that for a Rees matrix semigroup
$S$, $\ell^{1}(S)$ is $n$-weakly amenable when $n$ is odd. Also
he obtained a similar result for a regular semigroup $S$ with
finitely many idempotents.
\par
Let $\A$ and $\U$ be Banach algebras such that $\A$ is a Banach
$\U$-bimodule with compatible actions, that is
\[ \alpha .(ab)=(\alpha .a)b, \quad (ab).\alpha =a(b.\alpha )
\,\,\, (a,b\in \A, \alpha \in \U). \] Let $\X$ be a Banach
$\A$-bimodule and a Banach $\U$-bimodule with compatible actions,
that is
\[ \alpha . (ax)=(\alpha .a)x, \quad a(\alpha .x)=(a.\alpha
)x, \quad (\alpha . x)a=\alpha .(xa) \,\,\, (a\in \A, \alpha \in
\U, x\in \X), \] and similarly for the right or two-sided
actions. Then $\X$ is called a \emph{Banach $\A$-$\U$-module}, and
is called a \emph{commutative} Banach $\A$-$\U$-module whenever
$\alpha . x =x.\alpha $ for all $\alpha \in \U$ and $x\in \X$.
\par
Let $\A$ and $\U$ be as above and $\X$ be a Banach
$\A$-$\U$-module. A bounded map $D:\A\rightarrow \X$ is called a
\emph{module derivation} if
\[ D(a \pm b)=D(a) \pm D(b), \quad D(ab)=aD(b)+D(a)b \,\,\, (a,b\in \A),\]
and
\[ D(\alpha .a)=\alpha .D(a), \quad D(a. \alpha )=D(a). \alpha
\,\,\, ( a\in \A, \alpha \in \U). \] Note that $D$ is not
necessarily linear and if there exists a constant $M > 0$ such
that $\parallel D(a) \parallel\leq M\parallel a \parallel $, for
each $a\in \A$, then $D$ is bounded and its boundedness implies
its norm continuity. When $\X$ is a commutative Banach
$\A$-$\U$-module, each $x\in \X$ defines an $\U$-module derivation
\[ D_{x}(a)=ax-xa \,\,\, (a\in \A), \]
these are called \emph{inner} module derivations.
\par
If $\X$ is a (commutative) Banach $\A$-$\U$-module, then so is
$\X^{*}$, where the actions of $\A$ and $\U$ on $\X^{*}$ are
naturally defined as above. So by letting $\X^{(0)}=\X$, if we
define $\X^{(n)}$ ($n\in \mathbb{N}$) inductively by
$\X^{(n)}=(\X^{(n-1)})^{*}$, then $\X^{(n)}$ is a (commutative)
Banach $\A$-$\U$-module.
\par
Note that when $\A$ acts on itself by algebra multiplication, it
is not in general a Banach $\A$-$\U$-module, as we have not
assumed the compatibility condition $ a( \alpha . b)=(a.\alpha)b
\,\,\, ( a,b\in \A, \alpha \in \U) $. If we consider the closed
ideal $J$ of $\A$ generated by elements of the form
$(a.\alpha)b-a( \alpha . b)$ for $a,b\in \A, \alpha \in \U$, then
$J$ is an $\U$-submodule of $\A$. So the quotient Banach algebra
$\A / J$ is a Banach $\U$-module with compatible actions and
hence from definition of $J$, when $\A / J$ acts on itself by
algebra multiplication, it is a Banach $(\A / J) $-$\U$-module.
Therefore, $(\A / J) ^{(n)}$ ($n\in \mathbb{N}$) is a Banach $(\A
/ J )$-$\U$-module. In general $\A / J$ is not a commutative
$\U$-module. If $\A / J$ is a commutative $\U$-module, then $(\A
/ J)^{(n)}$ ($n\geq 0$) is a commutative Banach $(\A / J)
$-$\U$-module. Now it is clear when $\A$ is a commutative
$\U$-module, then $J=\{ 0 \}$ and hence by multiplication of $\A$
from both sides, $\A^{(n)}$ ($n\geq 0$) is a commutative Banach
$\A$-$\U$-module.
\par
Let the Banach algebra $\A$ be a Banach $\U$-module with
compatible actions. From the above observations, $(\A / J)^{(n)}$
($n\geq 0$) is a Banach $\A$-$\U$-module by the $\A$-module
actions $a\Phi=(a+J)\Phi$ and $\Phi  a=\Phi (a+J)$ for $a,b\in
\A, \Phi \in (\A / J)^{(n)}$ (the $\U$-module actions are similar
to actions on $(\A / J)^{(n)}$ as $\U$-module). Note that
whenever $\A / J$ is a commutative $\U$-module, then $(\A /
J)^{(n)}$ ($n\geq 0$) is a commutative Banach $\A$-$\U$-module by
the above actions. Now we are ready to define the notion of
$n$-weak module amenability. We say that $\A$ is \emph{n-weakly
module amenable} ($n\in \mathbb{N}$) if $(\A / J)^{(n)}$ is a
commutative Banach $\A$-$\U$-module, and each continuous module
derivation $D:\A\rightarrow (\A / J)^{(n)}$ is inner; that is
$D(a)=D_{\Phi}(a)=a\Phi -\Phi a$ for some $\Phi\in (\A /
J)^{(n)}$ and all $a\in \A$. Also $\A$ is called
\emph{permanently weakly module amenable} if $\A$ is $n$-weakly
module amenable for each $n\in \mathbb{N}$. This definition is
quite natural since $(\A / J)^{(n)}$ ($n\geq 0$) is always a
Banach $\A$-$\U$-module.
\par
The notion of weak module amenability of a Banach algebra $\A$
which is a Banach $\U$-module with compatible actions is defined
in \cite{am1} and studied in \cite{am2}. The main result of
\cite{am1} is that $\ell^{1}(S)$ is weakly module amenable, as an
$\ell^{1}(E)$-module, when $S$ is commutative. The definition of
weak module amenability is modified in \cite{am2} and the above
result is proved for an arbitrary inverse semigroup (with trivial
left action). Then the notion of $n$-weak module amenability is
introduced in \cite{bod} and proved that $\ell^{1}(S)$ is
$(2n+1)$-weakly module amenable as an $\ell^{1}(E)$-module, for
each $n\in \mathbb{N}$, where $S$ is an inverse semigroup with
the set of idempotents $E$.
\par
In this paper we show that the inverse semigroup algebra
$\ell^{1}(S)$ is $2n$-weakly module amenable as an
$\ell^{1}(E)$-module, for every number $n\in \mathbb{N}$, where
$E$ is the set of idempotents of $S$ and $E$ acts on $S$
trivially from the left and by multiplication from the right. Our
proof is based on a common fixed point property for semigroups
and the idea of our proof comes from \cite{zh}. In fact in this
article we show that a module version of the main result of
\cite{zh} holds for inverse semigroups.

\section{Main result}
A discrete semigroup $S$ is called an \emph{inverse semigroup} if
for each $s \in S$ there is a unique element $s^{*} \in S$ such
that $ss^{*}s = s$ and $s^{*}ss^{*} = s^{*}$. An element $e \in S$
is called an \emph{idempotent} if $e = e^{*} = e^{2}$. The set of
idempotents of $S$ is denoted by $E$. There is a natural order on
$E$, defined by
\[ e\leq d \Leftrightarrow ed=e \quad (e,d \in E), \]
and $E$ is a commutative subsemigroup of S, which is also a
semilattice \cite[Theorem V.1.2]{ho}. Elements of the form
$ss^{*}$ are idempotents of S and in fact all elements of $E$ are
in this form.
\par
The algebra $\ell^{1}(E)$ could be regarded as a subalgebra of
$\ell^{1}(S)$. Hence $\ell^{1}(S)$ is a Banach algebra and a
Banach $\ell^{1}(E)$-module with compatible actions. In this
article we let $\ell^{1}(E)$ act on $\ell^{1}(S)$ by
multiplication from right and trivially from left, that is
\[ \delta_{e}.\delta_{s}=\delta_{s}, \quad
\delta_{s} . \delta_{e}=\delta_{se}=\delta_{s}*\delta_{e} \quad
(s\in S, e\in E). \] In this case, the ideal $J$ (see section 1)
is the closed linear span of $\{\delta_{set}-\delta_{st}\, | \,
s,t\in S, e\in E \}$. With the notations of the previous section
$(\ell^{1}(S)/J) ^{(n)}$ ($n\geq 0$) is a Banach
$\ell^{1}(S)$-$\ell^{1}(E)$-module. Note that we show the
$\ell^{1}(E)$-module actions of $f\in \ell^{1}(E)$ on $\Phi \in
(\ell^{1}(S)/J) ^{(n)}$ by $f.\Phi$ and $\Phi .f$, and also denote
the $\ell^{1}(S)$ module actions of $f\in \ell^{1}(S)$ on $\Phi
\in (\ell^{1}(S)/J) ^{(n)}$ by $f\Phi$ and $\Phi f$. In the next
remark we give some properties of these module actions.
\begin{rem}\label{m1}
With the above notation, for all $e\in E$ and $\Phi\in
(\ell^{1}(S)/J) ^{(n)}$ $(n\geq 0)$ we have the followings
\begin{enumerate}
\item[(i)] $\delta_{e}. \Phi = \Phi .\delta_{e} $;
\item[(ii)] $\delta_{e}\Phi=\Phi\delta_{e}=\Phi$.
\end{enumerate}
\end{rem}
\begin{proof}
For all $e,d \in E$, we have
$\delta_{e}-\delta_{d}=\delta_{ee}-\delta_{ede}-\delta_{dd}+\delta_{ded}\in
J $. So $\delta_{e}+J=\delta_{d}+J$. Now for any $s\in S$ and $e
\in E$, we find
\[
\delta_{es}+J=(\delta_{e}+J)(\delta_{s}+J)=(\delta_{ss^{*}}+J)(\delta_{s}+J)=\delta_{s}+J.
\]
Similarly, we get $\delta_{se}+J=\delta_{s}+J$ for $e\in E$ and
$s\in S$. Hence
\[
\delta_{e}.(\delta_{s}+J)=\delta_{s}+J=\delta_{se}+J=(\delta_{s}+J).\delta_{e}
\]
and
\[\delta_{e}(\delta_{s}+J)=(\delta_{e}+J)(\delta_{s}+J)=\delta_{es}+J=\delta_{s}+J=\delta_{se}+J=(\delta_{s}+J)(\delta_{e}+J)=(\delta_{s}+J)\delta_{e},\]
for all $e\in E$ and $s\in S$. Since $lin\{\delta_{s}\, | \, s\in
S \}$ is dense in $\ell^{1}(S)$ and $J$ is closed in
$\ell^{1}(S)$, it follows that
\[\delta_{e}.(f+J)=(f+J).\delta_{e} \]
and
\[\delta_{e}(f+J)=f+J=(f+J)\delta_{e},\]
for all $e\in E$ and $f \in \ell^{1}(S)$. So by induction on $n$
we arrive at
\[\delta_{e}. \Phi = \Phi .\delta_{e} \]
and
\[ \delta_{e}\Phi=\Phi\delta_{e}=\Phi,\]
for all $e\in E$ and $\Phi\in (\ell^{1}(S)/J) ^{(n)}$ $(n\geq 0)$.
\end{proof}
In view of this remark (i), we find that $(\ell^{1}(S)/J) ^{(n)}$
$(n\geq 0)$ is a commutative $\ell^{1}(E)$-module.
\par
For an inverse semigroup $S$, the quotient $S / \approx $ is a
discrete group, where $\approx$ is an equivalence relation on $S$
as follows:
\[ s\approx t \Leftrightarrow \delta_{s}-\delta_{t}\in J \quad
(s,t\in S). \] Indeed, $S/ \approx $ is homomorphic to the maximal
group homomorphic image $G_{S}$ \cite{mun} of $S$ (see \cite{am3},
\cite{pou} and \cite{pou2}). As in \cite[Theorem 3.3] {rez}, we
may observe that $\ell^{1}(S)/J\cong \ell^{1}(G_{S})$. Also see
\cite{eb}.
\par
Since for proof of the main result we use a common fixed point
property for semigroups, now we recall some notions related to
common fixed point theory. Let $S$ be a (discrete) semigroup. The
space of all bounded complex valued functions on $S$ is denoted by
$\ell^{\infty}(S)$. It is a Banach space with the uniform
supremum norm. In fact $\ell^{\infty}(S) = (\ell^{1}(S))^{*}$.
For each $s \in S$ and each $f\in \ell^{\infty}(S)$ let
$\ell_{s}f$ be the left translate of $f$ by $s$, that is
$\ell_{s}f(t) = f(st)$ ($t \in S$) (the right translate $r_{s}f$
is defined similarly). We recall that $f \in \ell^{\infty}(S)$ is
\emph{weakly almost periodic} if its left orbit
$\mathcal{L}\mathcal{O}(f) = \{\ell_{s}f \, | \,  s \in S\}$ is
relatively compact in the weak topology of $\ell^{\infty}(S)$. We
denote by $WAP(S)$ the space of all weakly almost periodic
functions on $S$, which is a closed subspace of $\ell^{\infty}(S)$
containing the constant function and invariant under the left and
right translations. A linear functional $m \in WAP(S)^{*}$ is a
\emph{mean} on $WAP(S)$ if $\parallel m
\parallel = m(1) = 1$. A mean $m$ on $WAP(S)$ is a \emph{left
invariant mean} (abbreviated $LIM$) if $m(\ell_{s}f) = m(f)$ for
all $s\in S$ and all $f \in WAP(S)$. If $S$ is an inverse
semigroup, it is well known that $WAP(S)$ always has a $LIM$
\cite[Proposition 2] {dun}. Let $C$ be a subset of a Banach space
$\X$. We say that $\Gamma = \{ T_{s}\, | \, s \in S\}$ is a
\emph{representation} of $S$ on $C$ if for each $s \in S$, $T_{s}$
is a mapping from $C$ into $C$ and $T_{st}(x)=T_{s}(T_{t}(x))$
$(s, t \in S, x \in C)$. We say that $x\in C$ is a \emph{common
fixed point} for (the representation of) $S$ if $T_{s}(x) = x$ for
all $s \in S$.
\par
Let $\X$ be a Banach space and $C$ a nonempty subset of $\X$. A
mapping $T: C \rightarrow C$ is called \emph{nonexpansive} if
$\parallel T(x)- T(y)\parallel \leq \parallel x -y \parallel$ for
all $x, y \in C$. The mapping $T$ is called \emph{affine} if $C$
is convex and $T(\gamma x +\eta y)=\gamma T(x)+ \eta T(y)$ for
all constants $\gamma , \eta \geq 0$ with $\gamma + \eta = 1$ and
$x, y \in C$. A representation $\Gamma$ of a semigroup $S$ on
$C$ acts as nonexpansive affine mappings, if each $T_{s}$ $(s\in
S)$ is nonexpansive and affine.
\par
A Banach space $\X$ is called \emph{$L$-embedded} if there is a
closed subspace $\X_{0}\subseteq \X^{**}$ such that $\X^{**} = \X
\oplus_{\ell^{1}} \X_{0}$. The class of $L$-embedded Banach spaces
includes all $L^{1}(\Sigma, \mu)$ (the space of of all absolutely
integrable functions on a measure space $(\Sigma, \mu)$), preduals
of von Neumann algebras, dual spaces of $M$-embedded Banach spaces
and the Hardy space $H_{1}$. In particular, given a locally
compact group $G$, the space $L^{1}(G)$ is $L$-embedded. So are
its even duals $L^{1}(G)^{(2n)}$ $(n\geq 0)$. For more details we
refer the reader to \cite{zh} and the references therein. \par
The next lemma is the common fixed point theorem for semigroups in
\cite[Theorem 2]{zh}, which will be used in our proof to the main
result.
\begin{lem}\label{fix}
Let $S$ be a discrete semigroup and $\Gamma$ a representation of
$S$ on an $L$-embedded Banach space $\X$ as nonexpansive affine
mappings. Suppose that $WAP(S)$ has a $LIM$ and suppose that there
is a nonempty bounded set $B\subset \X$ such that $B \subseteq
\overline{ T_{s}(B)}$ for all $s\in S$, then $\X$ contains a
common fixed point for $S$.
\end{lem}
We now can prove the main result of the paper.
\begin{thm}
Let $S$ be an inverse semigroup with the set of idempotents $E$.
Consider $\ell^{1}(S)$ as a Banach module over $\ell^{1}(E)$ with
the trivial left action and natural right action. Then the
semigroup algebra $\ell^{1}(S)$ is $2n$-weakly module amenable as
an $\ell^{1}(E)$-module for each $n \in \mathbb{N}$.
\end{thm}
\begin{proof}
Let $D:\ell^{1}(S)\rightarrow (\ell^{1}(S)/J) ^{(2n)}$ be a
continuous module derivation. Since $ss^{*}\in E$ for all $s\in
S$, from Remark~\ref{m1}(ii), we have
\begin{equation*}
\begin{split}
D(\delta_{ss^{*}})&=D(\delta_{ss^{*}ss^{*}})=D(\delta_{ss^{*}}*\delta_{ss^{*}})\\&
=\delta_{ss^{*}}D(\delta_{ss^{*}})+D(\delta_{ss^{*}})\delta_{ss^{*}}\\&
=2D(\delta_{ss^{*}}).
\end{split}
\end{equation*}
Hence $D(\delta_{ss^{*}})=0$ for all $s\in S$. Define
$\phi:S\rightarrow (\ell^{1}(S)/J) ^{(2n)}$ by
\[ \phi(s)=D(\delta_{s})\delta_{s^{*}} \quad (s\in S). \]
We see that
\begin{equation}\label{e1}
\begin{split}
\phi(st)&=D(\delta_{s}*\delta_{t})\delta_{(st)^{*}}\\&
=(\delta_{s}D(\delta_{t}))\delta_{t^{*}}*\delta_{s^{*}}+(D(\delta_{s})\delta_{t})\delta_{t^{*}}*\delta_{s^{*}}\\&
=\delta_{s}(D(\delta_{t})\delta_{t^{*}})\delta_{s^{*}}+(D(\delta_{s})\delta_{tt^{*}})\delta_{s^{*}}\\&
=\delta_{s}(D(\delta_{t})\delta_{t^{*}})\delta_{s^{*}}+D(\delta_{s})\delta_{s^{*}}\\&
=\delta_{s}\phi(t)\delta_{s^{*}}+\phi(s),
\end{split}
\end{equation}
for all $s,t\in S$. Let $B=\phi(S)$. Then $B$ is a nonempty
bounded subset of $(\ell^{1}(S)/J) ^{(2n)}$. For any $s\in S$
define the mapping $T_{s}:(\ell^{1}(S)/J) ^{(2n)}\rightarrow
(\ell^{1}(S)/J) ^{(2n)}$ by
\[ T_{s}(\Phi)=\delta_{s}\Phi \delta_{s^{*}}+\phi(s) \quad
(\Phi\in (\ell^{1}(S)/J) ^{(2n)}).\] Clearly each $T_{s}$ $(s\in
S)$ is an affine mapping and for every $\Phi, \Psi \in
(\ell^{1}(S)/J) ^{(2n)}$ and $s\in S$ we have
\[
\parallel T_{s}(\Phi)-T_{s}(\Psi)\parallel=\parallel\delta_{s}\Phi
\delta_{s^{*}}+\phi(s)- \delta_{s}\Psi
\delta_{s^{*}}+\phi(s)\parallel  \leq \parallel \Phi - \Psi
\parallel .
\]
So each $T_{s}$ $(s\in S)$ is nonexpansive. Now by using
\eqref{e1} for any $s,t\in S$ and $\Phi, \Psi \in (\ell^{1}(S)/J)
^{(2n)}$ we find
\begin{equation*}
\begin{split}
T_{st}(\Phi)&=\delta_{st}\Phi \delta_{(st)^{*}}+\phi(st)\\&
=\delta_{s}(\delta_{t}\Phi\delta_{t^{*}})\delta_{s^{*}}+\delta_{s}\phi(t)\delta_{s^{*}}+\phi(s)\\&
=\delta_{s}T_{t}(\Phi)\delta_{s^{*}}+\phi(s)\\&
=T_{s}(T_{t}(\Phi)).
\end{split}
\end{equation*}
So $\Gamma=\{ T_{s}\, | \, s\in S\}$ defines a representation of
$S$ on $(\ell^{1}(S)/J) ^{(2n)}$ which is nonexpansive and affine.
From definition of $T_{s}$ and \eqref{e1}, for any $s,t\in S$ it
follows that $T_{s}(\phi(t))=\delta_{s}\phi(t)
\delta_{s^{*}}+\phi(s)=\phi(st)$. Therefore $T_{s}(B)\subseteq B$
$(s\in S)$. Let $\Phi \in B$. Now by Remark~\ref{m1}(ii) and the
fact that $D(\delta_{ss^{*}})=0$ $(s\in S)$, we have
\[
T_{s}(T_{s^{*}}(\Phi))=T_{ss^{*}}(\Phi)=\delta_{ss^{*}}\Phi\delta_{ss^{*}}+\phi(ss^{*})=\Phi
\quad (s\in S).\] Since $T_{s^{*}}(\Phi)\in B$, it follows that
$T_{s}(B)=B$ for each $s\in S$. Here $S$ is regarded as a discrete
semigroup.
\par
Since $\ell^{1}(S)/J\cong \ell^{1}(G_{S})$, where $G_{S}$ is the
maximal group homomorphic image, it follows that $(\ell^{1}(S)/J)
^{(2n)}$ is $L$-embedded. Also $WAP(S)$ has a $LIM$. So by
Lemma~\ref{fix}, there is $\Upsilon \in (\ell^{1}(S)/J) ^{(2n)}$
such that $T_{s}(\Upsilon)=\Upsilon$ for all $s\in S$, or
\[\delta_{s}\Upsilon \delta_{s^{*}}+\phi(s)=\Upsilon, \]
for all $s\in S$. So $\delta_{s}\Upsilon
\delta_{s^{*}}+D(\delta_{s})\delta_{s^{*}}=\Upsilon$ $(s\in S)$.
Hence
\begin{equation*}
D(\delta_{s})=\Upsilon\delta_{s}-\delta_{s}\Upsilon,
\end{equation*}
for all $s\in S$. By definition of left module action of
$\ell^{1}(E)$ on $\ell^{1}(S)$, we have
$\delta_{e}.\delta_{s}=\delta_{s}$ $(e\in E, s\in S)$. Since
$lin\{\delta_{s}\, | \, s\in S \}$ is dense in $\ell^{1}(S)$, we
find $\delta_{e}.f=f$ for all $e\in E$ and $f\in \ell^{1}(S)$.
Hence $\delta_{e}.(f+J)=f+J$ $(e\in E, f\in \ell^{1}(S))$.
Furthermore a routine inductive argument shows that for each
$e\in E$ and $\Phi \in (\ell^{1}(S)/J) ^{(2n)}$ $(n\geq 0)$, we
have $\delta_{e}.\Phi=\Phi$. From this result and the fact that
$D$ is a module mapping, for any $s \in S$ and $\lambda\in
\mathbb{C}$ we have
\begin{equation*}
\begin{split}
D(\lambda \delta_{s})&=D(\lambda \delta_{ss^{*}}.\delta_{s})\\&
=\lambda\delta_{ss^{*}}.D(\delta_{s})\\&
=\lambda\delta_{ss^{*}}.(\Upsilon\delta_{s}-\delta_{s}\Upsilon)\\&
=\lambda\delta_{ss^{*}}.(\Upsilon
\delta_{s})-\lambda(\delta_{ss^{*}}.\delta_{s})\Upsilon\\&
=\lambda(\Upsilon \delta_{s}-\delta_{s}\Upsilon).
\end{split}
\end{equation*}
Since $D$ is additive, we get $D( f)=\Upsilon f-f\Upsilon
$ for any $f \in \ell^{1}(S)$ of finite support. But $D$ is
continuous and functions of finite support are dense in
$\ell^{1}(S)$, hence
\[ D( f) =\Upsilon f-f\Upsilon=D_{(-\Upsilon)}(f) \quad (f\in
\ell^{1}(S)),\] therefore $D$ is inner. The proof is complete.
\end{proof}
In \cite{bod}, it has been proved that $\ell^{1}(S)$ is
$(2n+1)$-weakly module amenable as an $\ell^{1}(E)$-module, for
each $n\in \mathbb{N}$, where $S$ is an inverse semigroup with
the set of idempotents $E$. From this result and above theorem we
get the next corollary.
\begin{cor}
Let $S$ be an inverse semigroup with the set of idempotents $E$.
Consider $\ell^{1}(S)$ as a Banach module over $\ell^{1}(E)$ with
the trivial left action and natural right action. Then the
semigroup algebra $\ell^{1}(S)$ is permanently weakly module
amenable as an $\ell^{1}(E)$-module.
\end{cor}
With the notations in previous corollary, we have the next result.
\begin{cor}
Each continuous module derivation $D:\ell^{1}(S)\rightarrow
(\ell^{1}(G_{S})) ^{(n)}$ $(n\in\mathbb{N})$ is inner.
\end{cor}


\bibliographystyle{amsplain}

\begin{thebibliography}{20}

\bibitem{am1} M. Amini, D. Ebrahimi Bagha, Weak module amenability for
semigroup algebras, Semigroup Forum, 71 (2005), 18–-26.

\bibitem{am2} M. Amini and A. Bodaghi, Module amenability and weak module
amenability for second dual of Banach algebras, Chamchuri J.
Math. 2(1) (2010), 57-–71.

\bibitem{am3} M. Amini, A. Bodaghi, and D. Ebrahimi Bagha, Module amenability
of the second dual and module topological center of semigroup
algebras, Semigroup Forum 80 (2010), 302-–312.

\bibitem{bod} A. Bodaghi, M. Amini and R. Babaee, Module
derivations into iterated duals of Banach algebras, Proc.
Romanian Academy, series A, 12 (2011), 277--284.

\bibitem{choi} Y. Choi, F. Ghahramani, Y. Zhang, Approximate and pseudo-amenability of various classes
of Banach algebras, J. Funct. Anal. 256 (2009), 3158--3191.

\bibitem{dal} H. G. Dales, F. Ghahramani and N. Gr\o nb\ae k,
Derivations into iterated duals of Banach algebras, Stud. Math.
128 (1998), 19--54.

\bibitem{dun} J. Duncan, I. Namioka, Amenability of inverse semigroups and
their semigroup algebras. Proc. R. Soc. Edinb. A 80 (1988),
309-–321

\bibitem{eb} D. Ebrahimi Bagha and M. Amini, Module derivation problem for inverse
semigroups, Semigroup Forum 85 (2012), 525-–532.

\bibitem{ho} J. M. Howie, An introduction  to semigroup theory,
Academic Press, London. 1976.

\bibitem{john1} B. E. Johnson, Weak amenability of group algebras,
Bull. London Math. Soc., 23(3) (1991), 281--284.

\bibitem{john2} B. E. Johnson, Permanent weak amenability of group algebras of
free groups, Bull. London Math. Soc. 31 (5) (1999), 569-–573.

\bibitem{los} V. Losert, On derivation and crossed homomorphisms, in: Banach
Algebra 2009, in: Banach Center Pub., Vol. 91, Inst. Math., Pol.
Acad. Sci, Warszawa, 2010, 199--217.

\bibitem{los2} V. Losert, The derivation problem for group algebras, Ann. of
Math. 168 (2008), 221--246.

\bibitem{mew} O. T. Mewomo, On $n$-weak amenability of Rees
semigroup algebras, Proc. Indian Math. Sci., 118(4) (2008),
547--555.

\bibitem{mun} W. D. Munn, A class of irreducible matrix representations of an
arbitrary inverse semigroup, Proc. Glasgow Math. Assoc. 5 (1961),
41–-48.

\bibitem{pou} H. Pourmahmood-Aghababa, (Super) Module amenability, module
topological center and semigroup algebras, Semigroup Forum 81
(2010), 344–-356.

\bibitem{pou2} H. Pourmahmood-Aghababa, A note on two equivalence relations on
inverse semigroups, Semigroup Forum 84 (2012), 200–-202.

\bibitem{rez} R. Rezavand et al., Module Arens regularity for semigroup
algebras, Semigroup Forum 77 (2008), 300-–305.

\bibitem{zh} Y. Zhang, $2m$-weak amenability of group algebras, J.
Math. Anal. Appl. 396 (2012), 412--416.
\end{thebibliography}

\end{document}